\documentclass{gtmon_a}
\pdfoutput=1

\usepackage{amscd}

\proceedingstitle{Proceedings of the Nishida Fest (Kinosaki 2003)}
\conferencestart{28 July 2003}
\conferenceend{8 August 2003}
\conferencename{International Conference in Homotopy Theory}
\conferencelocation{Kinosaki, Japan}

\editor{Matthew Ando}
\givenname{Matthew}
\surname{Ando}

\editor{Norihiko Minami}
\givenname{Norihiko}
\surname{Minami}

\editor{Jack Morava}
\givenname{Jack}
\surname{Morava}

\editor{W Stephen Wilson}
\givenname{W Stephen}
\surname{Wilson}

\title[Homotopy theoretical considerations of Seiberg--Witten invariants]
{Homotopy theoretical considerations 
of the Bauer--Furuta\\stable homotopy Seiberg--Witten invariants}

\author{Mikio Furuta}
\givenname{Mikio}
\surname{Furuta}
\address{Department of Mathematical Sciences\\
University of Tokyo\\\newline
Tokyo 153-8914\\
Japan}
\email{furuta@ms.u-tokyo.ac.jp}
\urladdr{}

\author{Yukio Kametani}
\givenname{Yukio}
\surname{Kametani}
\address{Department of Mathematics\\
Keio University\\\newline
Yokohama 223-8522\\
Japan}
\email{kametani@math.keio.ac.jp}
\urladdr{}

\author{Hirofumi Matsue}
\givenname{Hirofumi}
\surname{Matsue}
\address{Department of Economics\\
Seijo University\\\newline
Tokyo 157-8511\\
Japan}
\email{matsue@seijo.co.jp}
\urladdr{}

\author{Norihiko Minami}
\givenname{Norihiko}
\surname{Minami}
\address{Department of Mathematics\\
Nagoya Institute of Technology\\\newline
Nagoya 466-8555\\
Japan}
\email{minami.norihiko@nitech.ac.jp}
\urladdr{}

\volumenumber{10}
\issuenumber{}
\publicationyear{2007}
\papernumber{8}
\startpage{155}
\endpage{166}

\doi{}
\MR{}
\Zbl{}

\arxivreference{}


\received{1 January 2005}
\revised{5 January 2006}
\accepted{}
\published{29 January 2007}
\publishedonline{29 January 2007}
\proposed{}
\seconded{}
\corresponding{}
\version{}

\makeatletter
\def\cnewtheorem#1[#2]#3{\newtheorem{#1}{#3}[section]
\expandafter\let\csname c@#1\endcsname\c@theorem}

\let\xysavmatrix\xymatrix
\def\xymatrix{\disablesubscriptcorrection\xysavmatrix}
\AtBeginDocument{\let\bar\wbar\let\tilde\widetilde\let\hat\what}

\newtheorem{theorem}{Theorem}[section]
\cnewtheorem{lemma}[theorem]{Lemma}
\cnewtheorem{proposition}[theorem]{Proposition}
\cnewtheorem{corollary}[theorem]{Corollary}
\cnewtheorem{question}[theorem]{Question}

\newtheorem*{H-S}{Theorem (Hovey--Sadofsky)}

\theoremstyle{definition}
\cnewtheorem{definition}[theorem]{Definition}
\cnewtheorem{example}[theorem]{Example}
\cnewtheorem{xca}[theorem]{Exercise}

\theoremstyle{remark}
\cnewtheorem{answer}[theorem]{Answer}
\cnewtheorem{hint}[theorem]{Hint}
\cnewtheorem{remark}[theorem]{Remark}
\cnewtheorem{notation}[theorem]{Notation}

\makeatother  

\numberwithin{equation}{section}

\newcommand{\la}{\langle}
\newcommand{\ra}{\rangle}

\newcommand{\Min}{ \operatorname{Min} }
\newcommand{\Max}{ \operatorname{Max} }

\newcommand{\Znn}{{\mathbb{Z}_{\geq 0}}}

\newcommand{\N}{\mathbb{N}}

\renewcommand{\J}{\mathcal{J}}

\newcommand{\HH}{\mathbb{H}}

\newcommand{\dotbox}{\hbox to 1em{\hss.\hss}}

\newcommand{\level}{{\rm level}}
\newcommand{\slevel}{{\rm slevel}}
\newcommand{\colevel}{{\rm colevel}}
\newcommand{\scolevel}{{\rm scolevel}}
\newcommand{\sign}{{\rm sign}}
\newcommand{\link}{{\rm link}}
\newcommand{\slink}{{\rm slink}}

\makeop{Spin}
\makeop{Pin}

\begin{document}

\begin{htmlabstract} 
<p class="noindent">
We show the "non-existence" results are
essential for all the previous known applications of 
the Bauer&ndash;Furuta stable homotopy Seiberg&ndash;Witten invariants.
As an example, we present a unified proof of the adjunction
inequalities.</p>

<p class="noindent">
We also show that the nilpotency phenomenon explains
why the Bauer&ndash;Furuta stable homotopy Seiberg&ndash;Witten invariants
are not enough to prove 11/8&ndash;conjecture.</p>
\end{htmlabstract}

\begin{abstract} 
We show the ``non-existence'' results are
essential for all the previous known applications of 
the Bauer--Furuta stable homotopy Seiberg--Witten invariants.
As an example, we present a unified proof of the adjunction
inequalities.

We also show that the nilpotency phenomenon explains
why the Bauer--Furuta stable homotopy Seiberg--Witten invariants
are not enough to prove 11/8--conjecture.
\end{abstract}

\maketitle

\section{Background}
\label{sec1}

Nowadays, there are many applications of
the Bauer--Furuta stable homotopy Seiberg--Witten
invariants (see the work of Bauer and Furuta \cite{B,BF,F2}). Amongst
of all, we single out the following:
\begin{enumerate}
\item[(1)] $10/8$--theorem (see Furuta \cite{F1}),
\item[(2)] Adjunction inequalities (see the work of Furuta, Kametani,
Matsue and Minami \cite{FKMaMi,FKMa,FKMi}),
\item[(3)] Constructions of spin $4$--manifolds without
Einstein metric and computations of their Yamabe invariants
(see the work of Ishida and LeBrun~\cite{IL2,IL1,L}).
\end{enumerate}

However, there is some ramification amongst the
proofs of these results. Actually, (1) is
proven by reducing to ``non-existence'' results,
whereas (2) and (3) are proven by reducing to
``non-triviality'' results. Furthermore,
the techniques employed to show these  
``non-existence'' and ``non-triviality'' 
results have been different so far,
and experts have regarded them as rather independent results.

In this paper, we unify these
two approaches and deliver the following message:
%
To apply the Bauer--Furuta stable homotopy Seiberg--Witten 
invariants 
in the direction of (1), (2), and (3), 
it suffices to prove the ``non-existence'' results
(which is the standard way to attack $11/8$--conjecture ever 
since Furuta's celebrated paper \cite{F1}).

Now, there are a couple of advantages of our approach.
First, we can generalize the known results
for (2) and (3) to slightly wider classes of $4$--manifolds.
Second, in the course of our proof, we can conceptually
recognize how the ``nilpotency phenomenon'' is
responsible for the fact why we can \emph{never} 
prove the $11/8$--conjecture affirmatively by the 
Bauer--Furuta stable homotopy Seiberg--Witten invariants.

The basic ingredient of the present paper was originally
announced in the JAMI conference on Geometry and Physics
at the Johns Hopkins University in March, 2002.
The authors would like to thank JAMI for its hospitality.
The fourth author was partially supported by Grant-in-Aid
for Scientific Research No.13440020, Japan Society for the
Promotion of Science.


\section{Level and the $11/8$--conjecture 
}
\label{sec2}

In this section we review the concept of level for 
free $\Z/2$--spaces, and recall the results of
Stolz \cite{S}, Furuta \cite{F2}, Furuta--Kametani \cite{FK}
and others in this terminology.

\begin{definition}[see Dai--Lam~\cite{DL} and Dai--Lam--Peng~\cite{DLP}]
For a free $\Z/2$--space 
$X,$ define the level of $X,$ which we denote by $\level(X),$ 
as follows:
\[ 
\level(X) := \text{the smallest
$n$ such that }\
[X, S^{n-1}]^{\Z/2} \neq \emptyset,
\]
where $S^{n-1}$ is endowed with the antipodal $Z/2$--action.
\end{definition}

\begin{example}
\label{levelexample}
{\rm (i)}\qua The classical Borsuk--Ulam theorem
states $\level(S^m) = m+1$.

{\rm (ii)}\qua (see Stolz~\cite{S})\qua
For $\R P^{2m-1} = S(\C^m)/\la i^2\ra$ with multiplication by $i,$
\[
\level( \R P^{2m-1} )=
\begin{cases}
m+1 & \text{if}\ m \equiv 0, 2 \mod 8  \\
m+2 & \text{if}\ m \equiv 1,3,4,5,7 \mod 8  \\
m+3 & \text{if}\ m \equiv 6 \mod 8  
\end{cases}
\]
\end{example}

The concept of level comes into the
picture of $4$--manifolds because of the following 
consequence of the Bauer--Furuta stable homotopy Seiberg--Witten
invariants for closed Spin $4$--manifolds (see Furuta \cite{F1},
Furuta--Kametani \cite{FK}, and \fullref{sec6}),
via the ``$G$--join theorem'' (see Minami~\cite{Min} and
Schmidt~\cite{Sch}): 

\begin{theorem} \label{CP}
If there is a closed Spin 4--manifold  
with $k = -\frac{\sign(X)}{16}$ and $l=b_2^+(X)$, then
\[
l \geq \level( \C P^{2k-1} ).
\]
Here, via the unit sphere $S(\HH^k)$ of the direct sum of
the $k$--copies of the quaternions $\HH,$
$\C P^{2k-1} = S(\HH^k)/\{\cos\theta + i\sin\theta \mid
0 \leq \theta\leq 2\pi \}$ is endowed with $\Z/2$--action 
induced by multiplication by $j\in\HH.$ 
\end{theorem}

Such a result is interesting because the celebrated
Yukio Matsumoto November the 8th birthday conjecture
\cite{Mats} (commonly called ``$11/8$--conjecture'')
predicts:
\[
  l \geq 3k
\]
Thus, the determination of $\level( \C P^{2k-1} )$
is clearly important. It is easy to see $\level( \C P^{2k-1} )
\leq 3k.$  On the other hand, since 
$\R P^{4k-1} = S(\HH^k)/\{\pm 1\},$ there is an
obvious $\Z/2$--map  $\R P^{4k-1} \to \C P^{2k-1},$ and so
\[
3k\geq \level( \C P^{2k-1} ) \geq \level( \R P^{4k-1} ) =
\begin{cases}
2k + 1 \quad &\text{if}\ k \equiv 1  \mod 4 \\
2k + 2 \quad &\text{if}\ k \equiv 2  \mod 4 \\
2k + 3 \quad &\text{if}\ k \equiv 3  \mod 4\\
2k + 1 \quad &\text{if}\ k \equiv 4 \mod 4
\end{cases}
\]
where we used Stolz' theorem.

Clearly, the interesting questions are: how large $\level( \C P^{2k-1} )$
could be, and whether $\level( \C P^{2k-1} )$  is ever shown 
to be $3k.$

However, John Jones realized
$\level( \C P^{2k-1} ) \neq 3k$ via explicit computations
for some small $k,$ and conjectured
\[
\level( \C P^{2k-1} ) =
\begin{cases}
2k + 2 \quad &\text{if}\ k \equiv 1 \mod 4,\ k>1 \\
2k + 2 \quad &\text{if}\ k \equiv 2 \mod 4 \\
2k + 3 \quad &\text{if}\ k \equiv 3 \mod 4\\
2k + 4 \quad &\text{if}\ k \equiv 4 \mod 4
\end{cases}
\]

Now the current best result in this direction is the following:

\begin{theorem}[Furuta--Kametani \cite{FK}]
\label{FK}
\[
\level( \C P^{2k-1} ) \geq 
\begin{cases}
2k + 1 \quad &\text{if}\ k \equiv 1 \mod 4 \\
2k + 2 \quad &\text{if}\ k \equiv 2 \mod 4 \\
2k + 3 \quad &\text{if}\ k \equiv 3 \mod 4\\
2k + 3 \quad &\text{if}\ k \equiv 4 \mod 4
\end{cases}
\]
\end{theorem}

We note that the Bauer--Furuta stable homotopy Seiberg--Witten
invariants are defined ``stably,'' 
(see \fullref{sec6}) and 
Furuta--Kametani \cite{FK} (also the original work of Furuta~\cite{F1})
actually proved the ``stable'' version (which is
actually equivalent to the above statement via the 
$G$--join theorem).

Whereas Stolz~\cite{S} and Furuta--Kametani~\cite{FK} are
``non-existence'' results, other kinds of applications of the
Bauer--Furuta stable homotopy Seiberg--Witten invariants (see
\fullref{sec1} for more details) require ``non-triviality'' results.

To unify ``non-existence'' and
``non-triviality'' approaches, we generalize
the concept of level in the next section.

\section{Level, colevel, and their stable analogues} \label{sec3}
In this section, we present some very general definitions.

\begin{definition} Fix a topological group $G$ and a non-empty
$G$--space $A,$ and let $X$ and $Y$ be arbitrary $G$--spaces.

{\rm (i)}\qua Denote the iterated join $*^{k}A$ inductively
so that $*^{0}A = A,\ *^{k}A = A*(*^{k-1}A).$
We also understand  $(*^{-1}A)*Y = Y.$ Then, set
$\link_{A}(X,Y),$ \emph{link} of $X$ to $Y$
with respect to $A,$ by  
\[
\link_{A}(X,Y) :=
\begin{cases}
\Min \{n\in \Znn \mid  [X, (*^{n-1}A)*Y]^G \neq \emptyset \} 
\quad &\text{if}\ [X,Y]^G \neq \emptyset; \\
- \Max \{n\in \N \mid  [(*^{n-1}A)*X, Y]^G \neq \emptyset \} 
\quad &\text{if}\ [X,Y]^G = \emptyset,  
\end{cases}
\]
where we set 
{\small
\[
\begin{cases}
\Min \{n\in \Znn \mid  [X, (*^{n-1}A)*Y]^G \neq \emptyset \} 
&= +\infty \\
&\text{if }
\{n\in \Znn \mid  [X, (*^{n-1}A)*Y]^G \neq \emptyset \} 
= \emptyset   \\
- \Max \{n\in \N \mid  [(*^{n-1}A)*X, Y]^G \neq \emptyset \} 
&=  -\infty \\
&\text{if }
 [(*^{n-1}A)*X, Y]^G \neq \emptyset\ \text{for any}\ 
n\in\N.
\end{cases}
\]
}
%
{\rm (ii)}\qua Furthermore, set
$\slink_{A}(X,Y),$ \emph{stable link} of $X$ to $Y$
with respect to $A,$ by  
\[
\slink_{A}(X,Y) := \lim_{q\rightarrow \infty}
\link\bigl( \big( *^{q-1} A \big)*X, \big( *^{q-1} A \big)*Y \bigr).
\]
{\rm (iii)}\qua For extreme cases, set the \emph{level},
\emph{stable level}, \emph{colevel}, and
\emph{stable colevel} with respect to $A$ by:
\begin{align*}
\level_{A}(X) &:= \link_{A}(X, A) + 1  &
\slevel_{A}(X) &:= \slink_{A}(X, A) + 1  \\
\colevel_{A}(Y) &:= -\link_{A}(A, Y) + 1  &
\scolevel_{A}(Y) &:= -\slink_{A}(A, Y) + 1
\end{align*}
\end{definition}

\begin{remark} \label{stable}
(i)\qua Clearly, $\slink_{A}(X,Y)\leq \link_{A}(X,Y).$
Consequently, $\slevel_{A}(X)\leq \level_{A}(X)$ and
$\scolevel_{A}(X)\geq \colevel_{A}(X).$

(ii)\qua When $G=\Z/2$ and 
$A = \Z/2$ with the free $\Z/2$--action,  
then $\level_{\Z/2}$ and $\colevel_{\Z/2}$ 
are respectively the classical level and the colevel 
in the sense of Dai and Lam \cite{DL} (and \fullref{sec2}). 
This is because
\[
*^{n-1}\Z/2 = *^{n-1}S(\R) = S(\R^n) = S^{n-1},
\]
where $\R$ with the sign representation, 
$S^{n-1}$ with the antipodal action.

(iii)\qua 
Suppose $X$ is a free $\Z/2$--space such that
\[
\dim X \leq 2\cdot \slevel_{\Z/2}(X) - 1,
\]
then
$
\slevel_{\Z/2}(X) = \level_{\Z/2}(X) = \level(X).
$

Actually, this is a direct consequence of
the $G$--join theorem (see Schmidt~\cite{Sch} and Minami~\cite{Min},
and examples satisfying this condition include
$S^{m}, \R P^{2m-1}, \C P^{2k-1}$ (cf Stolz~\cite{S} and Furuta~\cite{F1}).
\end{remark}

\section{Level and ``non-triviality''}

We begin with the fundamental question which relates
``non-triviality'' problem to the concept of level.

\begin{question} \label{question}
When $n:=\level_{G/H}(X),$ does the restriction
$[X, *^{n-1}G/H]^G \to [X, *^{n-1}G/H]^H$
ever hit a constant map?
\end{question}

\begin{remark} \label{qr}
(i)\qua When $G=\Z/2,\ H=\{e\}$ and
$n > \level_{G/H}(X),$ the non-empty image of the composite
\[
[X, *^{n-2}G/H]^G  \to [X, *^{n-1}G/H]^G  
\to [X, *^{n-1}G/H]^H
\]
consists of the constant maps.  In fact, this follows
immediately from the triviality of the bottom arrow
in the following commutative diagram:
\[
\begin{CD}
[X, *^{n-2}G/H]^G  @>>> [X, *^{n-1}G/H]^G  \\
@VVV   @VVV    \\
[X, *^{n-2}G/H]^H  @>>> [X, *^{n-1}G/H]^H   \\
@|    @|   \\
[X, S^{n-2}]  @>>> [X,S^{n-1}]
\end{CD}
\]
(ii)\qua
When $G=\Z/2,\ H=\{e\}$ and $X =S^{n-1}$ with the
antipodal $\Z/2$--action, as was remarked in
\fullref{levelexample}~(ii), the classical Borsuk--Ulam
theorem states 
$\level_{G/H}(X) = n.$ In this case, the other version of
the classical Borsuk--Ulam theorem states that
\[
\begin{CD}
[X, *^{n-1}G/H]^G  @>>> [X, *^{n-1}G/H]^H  \\
@|  @|   \\
[S^{n-1}, S^{n-1}]^{\Z/2} @>>> [S^{n-1}, S^{n-1}]
\end{CD}
\]
never hits a constant map. 

(iii)\qua
Suppose the restriction
\[
[X, *^{n-1}G/H]^G  
\to 
[X, *^{n-1}G/H]^H
\]
never hit the constant map for
$G=\Z/2,\ H=\{e\}$ and  $X =\C P^{2k-1}$ with the
$\Z/2$--action as in \fullref{CP}. Then the
Bauer--Furuta stable homotopy Seiberg--Witten invariant
(see \fullref{sec6}),
applied to closed Spin $4$--manifolds $M^4$ with 
$$b_1(M^4) = 0,\qquad k = - \frac{\sign(M^4)}{16},\qquad
b^+_2(M^4) = \level( \C P^{2k-1} ),$$
imply the
following geometric consequences for $M^4$: 

\begin{enumerate}
\item (adjunction inequality, see the work of Furuta, Kametani, Matsue and
Minami \cite{KM,FKMi,FKMaMi})\qua For any 
embedded oriented closed surface 
$\Sigma \subseteq M^4,$ 
\[
| 2g(\Sigma) - 2 | \geq [\Sigma]\cdot [\Sigma].
\]
\item (Ishida--LeBrun \cite{L,IL1,IL2})\qua Non-existence of
Einstein metrics and computations of
the Yamabe invariants under some circumstances.
\end{enumerate}
\end{remark}

In this way, the concept of level, which arises naturally
in ``non-existence'' problems, also show up
``non-triviality'' problems.

\section{Main Theorem}

We now state our main theorem, which partially answers
\fullref{question}.

\begin{theorem} \label{mt} 
{\rm(i)}\qua When $n:=\slink_{G/H}(X,Y),$
\begin{multline*}
\lim_{q\rightarrow \infty}
[X*\big( *^{q-1} G/H \big), 
Y*\big( *^{n+q-1} G/H \big)]^G \\[-1ex]
\longrightarrow 
\lim_{q\rightarrow \infty}
[X*\big( *^{q-1} G/H \big), 
Y*\big( *^{n+q-1} G/H \big)]^H
\end{multline*}
never hits a constant map.

{\rm(ii)}\qua When $n:=\link_{G/H}(X,Y)\leq 0,$
\[
[X*\big( *^{-n-1} G/H \big), Y]^G
\to 
[X*\big( *^{-n-1} G/H \big), Y]^H
\]
never hits a constant map.
\end{theorem}

\begin{corollary} {\rm (i)}\qua When $l:=\slevel_{G/H}(X),$
$$\lim_{q\rightarrow \infty}
[X*\big( *^{q-1} G/H \big), 
\big( *^{l+q-1} G/H \big)]^G
\to 
\lim_{q\rightarrow \infty}
[X*\big( *^{q-1} G/H \big), 
\big( *^{l+q-1} G/H \big)]^H$$
never hits a constant map.

{\rm (ii)}\qua When $c:=\colevel_{G/H}(Y),$
\[
[\big( *^{c-1} G/H \big), Y]^G
\to 
[\big( *^{c-1} G/H \big), Y]^H
\]
never hits a constant map.
\end{corollary}

\begin{proof}[Proof of \fullref{mt}]
For both (i) and (ii),
it suffices to show $[X'*(G/H),Y']^G \neq \emptyset$ 
if the restriction map
\[
[X',Y']^G \to [X',Y']^H  \cong [X'\times G/H, Y']^G
\]
hits a constant map.  But, this follows easily from the
following observations:

\begin{enumerate}
\item The restriction 
$ 
[X',Y']^G \to [X',Y']^H  \cong [X'\times G/H, Y']^G
$ 
is induced by the first projection
$ 
X'\times G/H \to X'.
$ 
\item
Any constant map in $[X',Y']^H$ corresponds to a map
factorizing the second projection
$ 
X'\times G/H \to G/H
$ 
in $[X'\times G/H, Y']^G.$
\item
There is a homotopy push-out diagram
\[
\begin{CD}
X'\times G/H @>>> G/H  \\
@VVV   @VVV \\
X'    @>{i}>>  X'*(G/H).
\end{CD}
\proved
\]
\end{enumerate}
\end{proof}

From \fullref{mt} and the $G$--join theorem \cite{Min,Sch}
(see \fullref{stable}~(ii)),
we obtain the following important consequence.

\begin{theorem} \label{consequence}
$G=\Z/2,\ H=\{e\},\ X:$
a free $\Z/2$--space such that
\[
\dim X \leq 2n - 1,
\]
where $n = \slevel_{\Z/2}(X).$
Then
\[
[X,S^{n-1}]^{\Z/2} \to [X,S^{n-1}]
\]
never hit a constant map.
\end{theorem}

\fullref{consequence} has the following applications:
\begin{enumerate}
\item It recovers the classical Borsuk--Ulam theorem (which was
quoted as ``the other version'' in the sense
of \fullref{qr}~(ii)).
\item It yields non-trivial family of elements of 
$\pi^*(\R P^{2m-1})$ together with Stolz' result.
\item It offers some applications to 4--manifolds via 
the Bauer--Furuta Seiberg--Witten invariants, along the
line of \fullref{qr}~(iii). 
\end{enumerate}
To give some flavor, we single out a statement about adjunction
inequalities, by applying Theorems~\ref{consequence} and~\ref{FK}:

\begin{theorem} \label{adjunction}
Let $M^4$ be a closed Spin $4$--manifolds $M^4$ with 
$b_1(M^4) = 0$ such that
\[
b^+_2(M^4) = 
\begin{cases}
- \frac{\sign(M^4)}{8} + 1 \quad &\text{if}\ 
- \frac{\sign(M^4)}{16} \equiv 1 \mod 4 \\
- \frac{\sign(M^4)}{8} + 2 \quad &\text{if}\ 
- \frac{\sign(M^4)}{16} \equiv 2 \mod 4 \\
- \frac{\sign(M^4)}{8} + 3 \quad &\text{if}\ 
- \frac{\sign(M^4)}{16} \equiv 3 \mod 4\\
- \frac{\sign(M^4)}{8} + 3 \quad &\text{if}\ 
- \frac{\sign(M^4)}{16} \equiv 4 \mod 4
\end{cases}
\]
Then, 
for any embedded oriented closed surface 
$\Sigma \subseteq M^4,$ 
\[
| 2g(\Sigma) - 2 | \geq [\Sigma]\cdot [\Sigma].
\]
\end{theorem}

We remark the case $- \frac{\sign(M^4)}{16} = 1,2,3$
were already treated in more direct ``non-triviality'' approach
(cf Kronheimer--Mrowka~\cite{KM}, and Furuta, Kametani, Matsue and Minami
\cite{FKMi,FKMaMi}).
On the other hand, we have shown it by reducing to a
``non-existence'' result: \fullref{FK}.



\section{Nilpotency rules!}
\label{sec6}

In this section, we explains the conceptual reason why
we can never prove the $11/8$--conjecture affirmatively
by use of the Bauer--Furuta stable homotopy Seiberg--Witten
invariants.

For this purpose, we briefly recall the Furuta--Bauer stable 
homotopy Seiberg--Witten invariant \cite{F2,BF}.
We begin with notations:


Let $M^4$ be an oriented closed 4--manifold with $b_1(M^4)=0,$
$c$ a $\Spin^c$--structure of $M^4,$ and
[$o_{M^4}$] an orientation of $H^+(M^4),$ the maximal positive definite 
subspace of $H^2(M^4,\R).$
Set
$$m := \frac{c_1(c)^2- \sign(M^4)}{8}, \qquad\qquad n :=b^+_2(M^4)$$
and \emph{assume $m\geq 0,$} by changing the orientation
of $M^4$ if necessary.


Then the Furuta--Bauer stable homotopy Seiberg--Witten invariant
$SW(M^4,c,o_{M^4})$ for the data $(M^4,c,o_{M^4})$ is defined so that
$$SW(M^4,c,o_{M^4}) \in \{S(\C^{m}), S({\R}^n)\}^{U(1)}  \\
:= \!\lim_{p,q\rightarrow \infty}
[S(\C^{p+m} \oplus {\R}^{q}),
 S(\C^{p}\oplus {\R}^{q+n})]^{U(1)}$$
Suppose the $\Spin^c$--structure $c$ comes from
a $\Spin$--structure $s,$ and set:
$$k := - \frac{\sign(M^4)}{16}, \qquad\qquad l :=b^+,$$
where we \emph{assume $k\geq 0,$} by changing
the orientation of $M^4,$ if necessary.
We also prepare the following representation theoretic
notations:
\begin{center}
\begin{tabular}{ll}
$\Pin_2$ & The closed subgroup of the quaternions
$\HH,$ generated by $j$ and \\
& $U(1) = \{ \cos\theta + i\sin\theta \mid 0\leq \theta < 2\pi \}$. \\
$\HH$ & The quaternions $\HH,$ regarded as 
a right $\Pin_2$--module by the right \\
& $\Pin_2 (\subset \HH )$ multiplication. \\
$\tilde{\R}$ & $\R$ regarded as a right
$\Pin_2$--module via the sign representation \\
& of $\{\pm 1\}\cong \Pin_2/ U(1)$.
\end{tabular}
\end{center}
Then the  Furuta--Bauer stable homotopy Seiberg--Witten
invariant $SW(M^4,s,o_{M^4})$ for the data $(M^4,s,o_{M^4})$ is defined 
so that
$$SW(M^4,s,o_{M^4}) \in \{S(\HH^{k}),
S({\tilde\R}^l)\}^{\Pin_2}
:= \lim_{p,q\rightarrow \infty}
[S(\HH^{p+k} \oplus \tilde{\R}^{q}),
 S(\HH^{p}\oplus \tilde{\R}^{q+l})]^{\Pin_2}$$
Furthermore, if we regard a $\Spin$--structure $s$ as
as $\Spin^c$--structure $c,$ then the forgetful map
via $U(1) \subseteq \Pin_2$ induces the natural
correspondence:
\begin{align*}
  \{S(\HH^{k}), S({\tilde\R}^l)\}^{\Pin_2} &\to
  \{S(\C^{2k}), S({\R}^l)\}^{U(1)}  \\
  SW(M^4,s,o_{M^4}) &\mapsto SW(M^4,c,o_{M^4})
\end{align*}
Now the key to relate the Bauer--Furuta stable homotopy
Seiberg--Witten invariants to our discussions in the
previous sections is the following observation,
which follows from Furuta's original work~\cite{F1}
(which showed $l\geq 2k+1$) and the
$G$--join theorem \cite{Min,Sch}:  If 
$\{S(\HH^{k}), S({\tilde\R}^l)\}^{\Pin_2}  \neq \emptyset,$
then
\begin{equation} \label{translation}
\begin{CD}
 \{S(\HH^{k}), S({\tilde\R}^l)\}^{\Pin_2} @>>>
  \{S(\C^{2k}), S({\R}^l)\}^{U(1)}  \\
@|   @|    \\
[X,S^{l-1}]^{\Z/2} @>>> [X,S^{l-1}],
\end{CD}
\end{equation}
where $X = \C P^{2k-1}$ with the $\Z/2$--action
as in \fullref{CP}.

Now we are ready to offer a conceptual explanation
why we can never prove the $11/8$--conjecture by
use of the Bauer--Furuta stable homotopy Seiberg--Witten 
invariants.  

We first note we should
show
\begin{equation} \label{hope}
l < 3k \implies
\{S(\HH^{k}), S({\tilde\R}^l)\}^{\Pin_2} = \emptyset.
\end{equation}
to prove the $11/8$--conjecture via the Bauer--Furuta
stable homotopy Seiberg--Witten invariants.
Then consider the following key commutative diagram,
whose horizontal arrows are induced by the $(N-1)$--fold
iterated join map:
\[
\begin{CD}
\{S(\HH^{k}), S({\tilde\R}^l)\}^{\Pin_2}
@>{*^{N-1}}>> 
\{S(\HH^{Nk}), S({\tilde\R}^{Nl})\}^{\Pin_2} \\
@VVV         
@VVV  \\
\{S(\C^{2k}), S({\R}^l)\}^{U(1)}
@>{*^{N-1}}>>
\{S(\C^{2Nk}), S({\R}^{Nl})\}^{U(1)}
\end{CD}
\]
Now, because of the Bauer--Furuta Seiberg--Witten invariants 
of a $K3$--surface with $\Spin$--structure, we see
\[
\{S(\HH^{1}), S({\tilde\R}^3)\}^{\Pin_2} \neq \emptyset
\]
Thus we apply the above key commutative diagram with
$k=1,l=3,N\gg0.$  Then, because of the
Nilpotency Theorem \cite{nilpotency}, the bottom
horizontal arrow is the trivial map for sufficiently
large $N$ (actually, straight-forward computations show
this is trivial for $N\geq 5$).  Therefore, 
the right vertical map hits the constant map. 

Now \eqref{translation} allows us to apply 
\fullref{consequence}, which implies
\[
\{S(\HH^{N}), S({\tilde\R}^{3N-1})\}^{\Pin_2}
           \neq \emptyset.
\]
Of course, in view of \eqref{hope},
this means a failure of proving the
$11/8$--conjecture via the Bauer--Furuta stable homotopy
Seiberg--Witten invariants.

\bibliographystyle{gtart}
\bibliography{link}

\end{document}